\documentclass[reqno,12pt]{amsart}

\usepackage{color} 
\usepackage{ifpdf}
\ifpdf
    \usepackage[pdftex]{graphicx}
    \usepackage[pdftex]{hyperref}
    \hypersetup{
        unicode=false,          
        pdftoolbar=true,        
        pdfmenubar=true,        
        pdffitwindow=false,     
        pdfstartview={FitH},    
        pdftitle={MCP Article},      
        pdfauthor={Michael Holst},   
        pdfsubject={Mathematics},    
        pdfcreator={Michael Holst},  
        pdfproducer={Michael Holst}, 
        pdfkeywords={PDE, analysis, mathematical physics}, 
        pdfnewwindow=true,      
        colorlinks=true,        
        linkcolor=red,          
        citecolor=blue,         
        filecolor=magenta,      
        urlcolor=cyan           
    }

\else
    \usepackage{graphicx}
    \usepackage{pstricks}
    
    \newcommand{\href}[2]{#2}
\fi

\usepackage{times}
\usepackage{amsfonts}

\usepackage{amsmath}
\usepackage{amsthm}
\usepackage{amssymb}
\usepackage{amsbsy}
\usepackage{amscd}
\usepackage{multirow}

\usepackage{enumerate}
\usepackage{verbatim}
\usepackage{subfigure}
\usepackage[algo2e,linesnumbered,ruled]{algorithm2e}
\usepackage{mcode}

\newtheorem{theorem}{Theorem}[section]
\newtheorem{corollary}[theorem]{Corollary}
\newtheorem{lemma}[theorem]{Lemma}

\newtheorem{remark}[theorem]{Remark}

\numberwithin{equation}{section}  

  \newcounter{mnote}
  \let\oldmarginpar\marginpar
    \renewcommand\marginpar[1]{\-\oldmarginpar[\raggedleft\footnotesize #1]%
    {\raggedright\footnotesize #1}}

\newcommand{\B}{{\mathbb B}}
\newcommand{\R}{{\mathbb R}}       
\renewcommand{\P}{{\mathbb P}}

\newcommand{\cB}{{\mathcal B}}

\newcommand{\cJ}{{\mathcal J}}
\newcommand{\cK}{{\mathcal K}}

\newcommand{\cR}{{\mathcal R}}

\newcommand{\cT}{{\mathcal T}}

\newcommand{\bs}{\boldsymbol}
\begin{document}

\title[Auxiliary space preconditioners for VEMs]
{Auxiliary space preconditioners for virtual element methods on polytopal meshes}

\author[Y. Zhu]{Yunrong Zhu}
\email{zhuyunr@isu.edu}

\address{Department of Mathematics \& Statistics\\
         Idaho State University\\ 
         Pocatello, ID 83209-8085}

\thanks{YZ was supported in part by NSF DMS 1319110.}

\date{\today}

\keywords{Virtual element methods, polytopal mesh, auxiliary space preconditioner, fictitious space preconditioner, jump coefficients}

\begin{abstract}
	 In this paper, we develop the auxiliary space preconditioners for solving the linear system arising from the virtual element methods discretization on polytopal meshes for the second order elliptic equations. The preconditioners are constructed based on an auxiliary simplicial mesh. The condition numbers of the preconditioned systems are uniformly bounded, independent of the problem size and the jump in coefficients. Several numerical experiments are presented to demonstrate the performance of the preconditioners.
\end{abstract}

\maketitle


%

\section{Introduction}
In this paper, we present the auxiliary space preconditioning techniques (\cite{Xu.J1996}) for the linear system arising from virtual element methods (VEMs) discretization on polytopal mesh for the following second order elliptic problems with jump coefficients: 
\begin{equation} \label{eqn:model}
	-\nabla \cdot (\kappa\nabla u) =f \mbox{  in  } \Omega, \qquad u|_{\partial \Omega} =0.
\end {equation}
Here, $\Omega\subset \R^{d}$ ($d=2, 3$) is an open polygonal domain and $f\in L^2(\Omega)$. We assume that the diffusion coefficient $\kappa \in L^{\infty}(\Omega)$ is piecewise constant with respect to an initial polytopal partition $\cT_{0}$ of $\Omega$ but may have large jumps across the interface of the partition.  

Recently introduced in \cite{Beirao-da-Veiga.L;Brezzi.F;Cangiani.A;Manzini.G2013}, the VEMs have been widely used in the numerical approximation of various partial differential equations (PDEs) (see for example \cite{Antonietti.P;Veiga.L;Scacchi.S;Verani.M2016,Beirao-da-Veiga.L;Brezzi.F;Marini.L;Russo.A2016,Cangiani.A;Chatzipantelidis.P;Diwan.G;Georgoulis.E2017,Chi.H;Veiga.L;Paulino.G2017}). The VEMs are generalizations of the classical finite element methods (FEMs), which permit the use of general polygonal and polyhedral meshes. Using polytopal meshes allows for more flexibility in dealing with complex computational domains or interfaces (cf. \cite{Chen.L;Wei.H;Wen.M2017}). Some error estimates of the VEM are recently given in \cite{Brenner.S;Guan.Q;Sung.L2017,Chen.L;Huang.J2018}. For practical implementation of the the VEM for second order elliptic equations, we refer to \cite{Beirao-da-Veiga.L;Brezzi.F;Marini.L;Russo.A2014,Sutton.O2017}. 

On the other hand, little work has been devoted to design efficient solvers for the VEM. In \cite{Bertoluzza.S;Pennacchio.M;Prada.D2017}, the authors developed and analyzed balancing domain decomposition preconditioners for elliptic problems discretized by the virtual element method.  In \cite{Antonietti.P;Mascotto.L;Verani.M2017}, the authors presented a $p$-version multigrid algorithm for the VEM discretization of 2D Poisson problem. The coarse spaces are constructed by decreasing progressively the polynomial degree of the VEM space. It was shown that the multigrid algorithm converges uniformly. 

In order to design an efficient, robust and easily implementable preconditioner for VEM on polytopal mesh, we use the general auxiliary space preconditioning framework developed in \cite{Xu.J1996}. This framework allows us to construct preconditioners in the lack of the hierarchical subspaces that is required in the classical (geometric) multilevel methods. The prominent examples of this framework include \cite{Hiptmair.R;Xu.J2007} for the Maxwell's equations and \cite{Chen.L;Wang.J;Wang.Y;Ye.X2015} for the weak Galerkin method on simplicial mesh.  The preconditioners we will develop here can be understood as two-level algorithms. 
The ``fine'' level problem is the VEM discretization on general polytopal mesh, and the ``coarse'' level problem is the standard conforming $\P_{1}$ finite element space on an auxiliary simplicial mesh. It is natural to choose the standard $\P_{1}$ finite element spaces as the coarse space for a couple of reasons: (1) the degrees of freedom in the coarse space are included in the VEM space -- so asymptotically, the solutions on the coarse and fine spaces should be close to each other; (2) there is a lot of literature on developing efficient (and robust) solvers for the standard conforming $\P_{1}$ finite element discretization, so we can use any existing solvers/preconditioners as a coarse solver. 

In this paper, we show that the auxiliary space preconditioners are robust with respect to both mesh size and the jumps in the coefficient. In addition, we show that the \emph{fictitious space preconditioner} using the conforming $\P_{1}$ finite element space as the fictitious space also provides an efficient and robust preconditioner for the linear VEM discretization. We also remark that these preconditioners are very easy to implement. In particular, the matrix form of the intergrid transfer operator between the linear VEM and the coarse space is just the identity matrix.

The rest of this paper is organized as follows. In Section~\ref{sec:vem}, we give basic notation and the virtual element discretization. Then in Section~\ref{sec:aux}, we present the auxiliary space preconditioners and discuss its convergence. In Section~\ref{sec:num}, we present several numerical experiments in both 2D and 3D to verify the theoretical results. Finally, in Section~\ref{sec:con} we summarize the results and discuss some possible extensions of these results. Throughout the paper, we will use the notation $x_1\lesssim y_1$, and $x_2\gtrsim y_2$,
whenever there exist constants $C_1, C_2$ independent of the mesh size
$h$ and the coefficient $\kappa$ or other parameters that $x_1$,
$x_2$, $y_1$ and $y_2$ may depend on, and such that $x_1 \le C_1 y_1$
and $x_2\ge C_2 y_2$.

\section{Virtual Element Methods}
\label{sec:vem}
Given any subset $G\subset \R^{d}$, we use standard notation $L^{2}(G)$ for the set of square integrable functions, with the inner product $(\cdot, \cdot)$ and the induced norm $\|\cdot\|_{0,G}$. We use the standard definition and notation of Sobolev spaces  $H^{1}(G)$ and $H_{0}^{1}(G)$. The associated inner product, norm, and semi-norm in $H^{1}(G)$ are denoted by $(\cdot, \cdot)_{1, G}$, $\|\cdot\|_{1,G}$, and $|\cdot|_{1,G}$, respectively. We will also use $(\cdot, \cdot)_{0, \kappa, G}$ to denote the weighted $L^{2}$-inner product
\[
	(u, v)_{0,\kappa, G}:=\int_{G} \kappa(\bs x) u v d\bs x,
\] and $\|\cdot\|_{0, \kappa, G}$ as the weighted $L^{2}$-norm. For simplicity, when $G=\Omega$, we omit it from the norms (or the inner products).  
The variational formulation of the equation \eqref{eqn:model} reads: find $u\in V := H_{0}^{1}(\Omega)$ such that 
\begin{align}
\label{eqn:weak}
	a(u,v) :=\int_{\Omega} \kappa \nabla u \cdot \nabla v dx = (f, v), \qquad \forall v\in V.
\end{align}

Let $\cT_{h}$ be a family of partitions of $\Omega$ into non-overlapping simple polytopal elements $K$. 
Here the subscript $h= h(K)$ denotes the diameter of the element $K$.  Following \cite{Chen.L;Huang.J2018}, we make the following assumption on the  polytopal mesh:
\begin{enumerate}
	\item[({\bf A})] Each polytopal element $K\in \cT_{h}$ has a ``virtual triangulation'' $\cT_{K}$ of $K$ such that $\cT_{K}$ is uniformly shape regular and quasi-uniform. Each edge of $K$ is an edge of certain element in $\cT_{K}$. 
\end{enumerate}
Let $k\ge 1$ be an integer, and $\P_{l}(D)$ be the polynomial of degree less than or equal to $l$ on the domain $D$. For convenience, we denote $\P_{-1}(D) =\{0\}.$ On each polytopal element $K\in \cT_{h}$, we define the local virtual finite element space: 
\[
	V_{h}^{K} :=\{ v\in H^{1}(K) \; : \; v|_{\partial K} \in \B_{k}(\partial K),\; \Delta v \in \P_{k-2}(K)\}, 
\]
where $\B_{k}(\partial K):= \{ v\in C^{0}(\partial K): v|_{e} \in \P_{k}(e), \; \forall e\subset \partial K\}.$
It is clear that $V_{h}^{K} \supset \P_{k}(K)$, and it may implicitly contain some other non-polynomial functions. 
The global virtual element space $V_{h}$ is then defined as: 
\[
	V_{h} := \{v\in V  \; : \; v|_{K} \in V_{h}^{K}, \;\; \forall K\in \cT_{h}\}. 
\]
In the 2D case, any function $v\in V_{h}$ can be uniquely determined by the following degrees of freedom: 
\begin{itemize}
	\item the function values of $v$ at the vertices of $\cT_{h}$;
	\item the function values of $v$ at the $k-1$ internal points of the $(k+1)$ Gauss-Lobatto quadrature rule on each edge $e$;
	\item the moments up to order $k-2$ of $v$ on each element $K\in \cT_{h}$.
\end{itemize}
For a more detailed construction of the $V_{h}$ in both 2D and 3D, we refer to \cite{Beirao-da-Veiga.L;Brezzi.F;Marini.L;Russo.A2014}. 
The following \emph{inverse inequality} was shown in \cite[Theorem 3.6]{Chen.L;Huang.J2018} (see also \cite[Lemma 2.19]{Brenner.S;Guan.Q;Sung.L2017}). 
\begin{lemma}
\label{lem:inverse}
	There exists a constant $C$ depending only on the shape regularity and quasi-uniformity of $\cT_{K}$ such that 
	$$
		\|\nabla v\|_{0,K} \le Ch_{K}^{-1} \|v\|_{0,K}, \qquad \forall v \in V_{h}^{K}.
	$$
\end{lemma}
The VEM discretization is designed in such as way that the ``virtual'' basis functions are not necessary to be computed explicitly. Instead, they are understood through the degrees of freedom. Although the ``virtual'' functions $v\in V_{h}^{K}$ are defined implicitly, we can compute their Galerkin projection $\Pi_{h}^{\nabla}:V_{h}^{K} \to \P_{k}(K)$ onto the polynomial space $\P_{k}(K)$. That is, 
\[
	(\nabla \Pi_{h}^{\nabla} v, \nabla p)_{K} = (\nabla v, \nabla p)_{K}, \qquad \forall p\in \P_{k}(K).
\]
$\Pi_{h}^{\nabla}v$ can be uniquely determined by an additional constraint, e.g., by setting 
\begin{align*}
	 \sum_{i=1}^{n_{v}^{K}} v(\bs x_{i}) &=\sum_{i=1}^{n_{v}^{K}} \Pi_{h}^{\nabla}v(\bs x_{i}), \qquad\mbox{ when } k=1,\\
	\int_{K} \Pi_{h}^{\nabla} v d \bs x &= \int_{K} v d\bs x, \qquad \mbox{ when } k\ge 2,
\end{align*}
where $n_{v}^{K}$ is the total number of vertices of $K$, and $\bs x_{i}$ ($i=1,\cdots, n_{v}^{K}$) are the coordinates of the vertices. 

Like the standard finite element methods, the VEM approximation of the equation \eqref{eqn:weak} is to find $u_{h} \in V_{h}$ such that 
\[
	a(u_{h}, v) = \sum_{K} a^{K}(u_{h}, v) = (f, v), \qquad \forall v \in V_{h},
\]
where $a^{K}(\cdot, \cdot):=a(\cdot, \cdot)|_{K}$ is the restriction of the bilinear form on $K$. 
For any $v\in V_{h}^{K}$, we may decompose $v = \Pi_{h}^{\nabla} v + (I - \Pi_{h}^{\nabla}) v.$ By the orthogonality of $\Pi_{h}^{\nabla}$ and $(I - \Pi_{h}^{\nabla})$ in the energy norm, we have 
$$
	a^{K}(u_{h}, v) = a^{K}(\Pi_{h}^{\nabla} u_{h}, \Pi_{h}^{\nabla} v) + a^{K}((I - \Pi_{h}^{\nabla}) u_{h}, (I - \Pi_{h}^{\nabla})v).
$$
Now the first term can be evaluated exactly. For the second term, we replace it by some computable bilinear form $a^{K}((I - \Pi_{h}^{\nabla}) u_{h}, (I - \Pi_{h}^{\nabla})v) \approx s^{K} ((I - \Pi_{h}^{\nabla}) u_{h}, (I - \Pi_{h}^{\nabla})v)$, to get
\[
	a_{h}^{K}(u_{h}, v):= a^{K}(\Pi_{h}^{\nabla} u_{h}, \Pi_{h}^{\nabla} v) + s^{K}((I - \Pi_{h}^{\nabla}) u_{h}, (I - \Pi_{h}^{\nabla})v).
\]
The first term in the definition of $a_{h}^{K}(\cdot, \cdot)$ is called \emph{consistency term}, and the second term is called the \emph{stabilizing term}. For the well-posedness and convergence of the VEM, it is important that $a_{h}^{K}(\cdot, \cdot)$ satisfies 
\begin{itemize}
	\item \emph{Polynomial consistency}: for any $v\in V_{h}^{K}$ and $p\in \P_{1}(K)$, it satisfies 
	\begin{equation}
	\label{eqn:consist}
		a_{h}^{K} (v, p) = a^{K}(v, p).
	\end{equation}
	\item \emph{Stability}: There exist constants $c_{1}$ and $c_{2}$ independent of $\kappa$ and $h$ such that 
	\begin{equation}
	\label{eqn:bstab}
		c_{1} a^{K}(v, v)  \le  a_{h}^{K}(v,v) \le c_{2} a^{K}(v, v), \qquad \forall v\in V_{h}^{K}.
	\end{equation}
\end{itemize}
There are different choices for the bilinear form $s^{K}(\cdot, \cdot)$ that satisfy the polynomial consistency \eqref{eqn:consist} and stability \eqref{eqn:bstab} conditions, see for example~\cite{Beirao-da-Veiga.L;Brezzi.F;Cangiani.A;Manzini.G2013,Beirao-da-Veiga.L;Brezzi.F;Marini.L;Russo.A2016}. In the numerical test performed in Section~\ref{sec:num}, we use the standard choice of $s^{K}$ as in~\cite{Beirao-da-Veiga.L;Brezzi.F;Cangiani.A;Manzini.G2013} with weights $\kappa_{K}= \kappa(x)|_{K}$ on each element $K\in \cT_{h}.$

Let $a_{h}(\cdot, \cdot) :=\sum_{K\in \cT_{h}} a_{h}^{K}(\cdot, \cdot).$ Then the VEM discretization of \eqref{eqn:weak} reads: find $u_{h}\in V_{h}$ such that 
\begin{equation}
\label{eqn:discrete}
	a_{h}(u_{h}, v_{h}) = \langle f_{h}, v_{h}\rangle, \qquad \forall v_{h}\in V_{h},
\end{equation}
where $f_{h}$ is the $L^{2}$ projection of $f$ on the piecewise $\P_{k-2}(K)$ for each element  $K\in \cT_{h}.$ 
Let $A$ be the operator induced by the bilinear form $a_{h}(\cdot, \cdot)$, that is, 
$$
	(A v, w) = (v, w)_{A} :=a_{h}(v, w), \qquad \forall v, w \in V_{h}.
$$
Then solving \eqref{eqn:discrete} is equivalent to solve the linear system 
\begin{equation}
\label{eqn:linear}
	Au_{h} = b.
\end{equation}
It is clear that the operator $A$ is symmetric and positive definite. Now we give an estimate of the condition number of $A$.

\begin{lemma}
\label{lem:spectralradius}
  The condition number of the operator $A$, denoted by $\cK(A)$, satisfies 
  $$
  	\cK(A) \lesssim \cJ(\kappa) h^{-2},
  $$
  where $\cJ(\kappa):= \max_{\bs x \in \Omega} \kappa(x) /\min_{\bs x \in \Omega} \kappa(x)$ is the jump in the coefficient. 
\end{lemma}
\begin{proof}
	By the stability of $a_{h}^{K}$ and inverse inequality Lemma~\ref{lem:inverse}, we have
	\begin{align*}
		(A v,v) & = a_{h}(v, v) = \sum_{K\in \cT_{h}} a_{h}^{K} (v,v) \\
		&\le \sum_{K\in \cT_{h}} c_{2} a^{K}(v,v) \le c_{2}\max_{\bs x \in \Omega} \kappa(\bs x) \sum_{K\in \cT_{h}}\|\nabla v\|^{2}_{0,K} \\
		&\lesssim h^{-2} \max_{\bs x \in \Omega} \kappa(\bs x) (v,v).
	\end{align*}
Since $A$ is symmetric and positive definite with respect to $(\cdot,\cdot)$, the above inequality implies that $\lambda_{\max}(A)\lesssim h^{-2} \max_{\bs x \in \Omega} \kappa(\bs x)$. 

On the other hand, by the Poincar\'e inequality and the stability of $a_{h}^{K}$, we obtain
	\begin{align*}
		 \min_{\bs x \in \Omega} \kappa(\bs x) (v,v) & \lesssim  \min_{\bs x \in \Omega} \kappa(\bs x) \|\nabla v\|^{2}\\
		&\lesssim \sum_{K\in \cT_{h}} a^{K}(v,v) \\
		&\lesssim  (A v, v),\qquad \forall v\in V_{h}.
	\end{align*}
This implies $\lambda_{\min}(A) \gtrsim 1/ \min_{\bs x \in \Omega} \kappa(\bs x)$. Therefore 
\[\cK(A) = \lambda_{\max}(A)/\lambda_{\min}(A)\lesssim h^{-2}\cJ(\kappa).\]
This completes the proof. 
\end{proof}
This lemma implies that the linear algebraic system resulting from the VEM discretization \eqref{eqn:discrete} is ill-conditioned with the condition number depends on both the mesh size and the jump in the coefficient. This dependence can be observed from the numerical experiments presented in Section~\ref{sec:num}. Thus the linear system \eqref{eqn:linear} is difficult to solve using the classic iterative methods such as Jacobi, Gauss-Seidel or CG, without effective preconditioners. The main purpose in this paper is the develop efficient auxiliary space preconditioners for the linear system \eqref{eqn:linear}.
\begin{remark}
	Based on the norm equivalences obtained in~\cite{Chen.L;Huang.J2018}, the results in this paper can be applied to the other type of VEM discretizations directly (namely, with different form of stabilization term).
\end{remark}


\section{Auxiliary Space Preconditioner}
\label{sec:aux}
In this section, we construct an auxiliary space preconditioner for solving the discrete system of equation \eqref{eqn:linear}. To do that, we need to introduce an auxiliary space. For each $K \in \cT_{h}$, we introduce an auxiliary triangulation such that each edge of $K$ is an edge of some element in this triangulation. This can be done using the Delaunay triangulation. This leads to a conforming triangulation of the whole domain $\Omega.$ Moreover, with the Assumption ({\bf A}) on $\cT_{h}$, the resulting triangulation is quasi-uniform. On this triangulation, we define a standard conforming $\P_{1}$ finite element subspace $V^{c}_{h}$. We consider the auxiliary problem: find $u^{c}_{h}\in V^{c}_{h}$ such that 
\begin{equation}
\label{eqn:aux}
	a(u^{c}_{h}, v_{h}) = (f, v_{h}), \qquad \forall v_{h}\in V^{c}_{h}.
\end{equation}
Similiary, let $A_{c}$ be the operator induced by the bilinear form $a(\cdot, \cdot)$, that is, 
$$
	(A_{c} v, w) = (v, w)_{A_{c}} :=a(v, w), \qquad \forall v, w \in V^{c}_{h}.
$$
The operator $A_{c}$ is symmetric and positive definite.

To construct the auxiliary space preconditioner, we use $V_{h}$ as the ``fine'' space and $V_{h}^{c}$ as the ``coarse'' space. Since $A_{c}$ is the conforming piecewise linear finite element discretization of the  equation \eqref{eqn:weak}, the ``coarse'' problem in $V_{h}^{c}$ can be solved by many existing efficient solvers such as the standard multigrid, domain decomposition  or algebraic multigrid (AMG) solvers (see for example \cite{Xu.J;Zhu.Y2008,Zhu.Y2008}).  Next, on the fine space, we define a ``smoother'' $\cR:\: V_{h}\rightarrow V_{h}$, which is symmetric positive definite. For example, $\cR$ could be a Jacobi or symmetric Gauss-Seidel smoother. We denote $s(\cdot, \cdot)$ as the bilinear form corresponding to $\cR^{-1}$.  Finally, to connect the ``coarse'' space $V_{h}^{c}$ with the ``fine'' space $V_{h}$, we need a ``prolongation'' operator $\Pi: V_{h}^{c} \to V_{h}$. A ``restriction'' operator $\Pi^{t} : V_{h} \rightarrow V_{h}^{c}$ is then defined as 
$$
	(\Pi^{t} v, w) = (v, \Pi w), \quad\textrm{for } v\in V_h\textrm{ and } w\in V_h^{c}.
$$
Then, the auxiliary space preconditioner $B:\: V_h\rightarrow V_h$ is given by 
\begin{align}
&\textrm{Additive} \qquad &&B_{\rm add} = \cR + \Pi A_{c}^{-1} \Pi^t,  \label{eqn:Badd}\\
&\textrm{Multiplicative} \qquad && I-B_{\rm mul}A = (I-\cR A)(I-\Pi A_{c}^{-1} \Pi^t)(I-\cR A). \label{eqn:Bmul}
\end{align}
With the notation introduced above, we have the following theorem for the auxiliary space preconditioners. 
\begin{theorem}[cf. \cite{Xu.J1996,Hiptmair.R;Xu.J2007}]
\label{thm:aux0}
Let the auxiliary space preconditioner defined above satisfy the following conditions: 
	\begin{itemize}
		\item[({\bf C0})] The smoother $\cR$ satisfies that 
		\[
			a_{h}(v, v) \le c_{0} s(v, v), \qquad \forall v\in V_{h},  
		\]
		where $c_{0} >0$ is independent $h$ and $\kappa$.
		\item[({\bf C1})] The operator $\Pi : V_{h}^{c} \to V_{h}$ is stable in the sense that there exists a constant $c_{1}>0$ independent of $h$ and $\kappa$ such that 
		\[
			a_{h}(\Pi  w, \Pi  w) \le c_{1} a (w, w), \qquad \forall w\in V_{h}^{c}.
		\]
		\item[({\bf C2})] For any $v\in V_{h}$, there exist $w\in V_{h}^{c}$ and $v_{0}\in V_{h}$ such that $v = v_{0}+ \Pi w$ such that 
		\[
			s(v_{0}, v_{0}) + a(w,w) \le c_{2} a_{h}(v, v).
		\]
	\end{itemize}
	Then 
	\[
		\cK(BA) \le c_{2}(c_{0} + c_{1}). 
	\]
\end{theorem}
\begin{remark}
\begin{enumerate}
	\item In the definition of the auxiliary space preconditioners \eqref{eqn:Badd}-\eqref{eqn:Bmul}, the coarse space solver $A_{c}^{-1}$ could be replaced by an inexact solver $\cB_{c} \approx A_{c}^{-1}$ (see for example \cite[Corollary 2.3]{Hiptmair.R;Xu.J2007}). For simplicity, we use the exact solver in the analysis. 
	\item In ~\eqref{eqn:Badd}, if we ignore the smoother $\cR$, then the preconditioner is usually called \emph{fictitious space preconditioner} (cf. \cite{Nepomnyaschikh.S1992}). In the sequel, we denote 
	\[
		B_{{\rm fict}}:= \Pi A_{c}^{-1} \Pi^t.
	\]
	The auxiliary space preconditioner can be viewed as a generalization of the fictitious space preconditioner by a special choice of the ``fictitious space''. In particular, the fictitious space is defined as a product space having $V_{h}$ itself as one of the component. This makes it easier to construct the map from the fictitious space to the original space, which is required to be surjective. 
\end{enumerate}
\end{remark}
The proof of
Theorem~\ref{thm:aux0} amounts to verifying three conditions ({\bf C0})-({\bf C2}). The smoothing property ({\bf C0}) is given by the following lemma. The proof of the smoothing properties are standard. The proof can be carried out by showing the estimates for the Jacobi smoother, as the Jacobi smoother and the symmetric Gauss-Seidel smoother are spectrally equivalent for any symmetric positive definition matrix (see for example \cite[Proposition~6.12]{Vassilevski.P2008} or \cite[Lemma~3.3]{Zikatanov.L2008}.) We omit the details here.
\begin{lemma}\label{lm:smoother}
  Let $s(\cdot,\cdot)$ be the bilinear form associated to Jacobi, or symmetric Gauss-Seidel smoother. Then  for any $v\in V_{h}$, we  have the following estimates
\begin{align}
a_{h}(v,v) &\le c_{0} s(v,v), \label{eqn:smoother1}\\
s(v,v) &\simeq h^{-2}\|v\|_{0, \kappa}^{2}. \label{eqn:smoother2}
\end{align}
Here, the constants are independent of coefficient and mesh size.
\end{lemma}

For the intergrid transfer operator $\Pi : V_{h}^{c} \to V_{h},$ we defined  on each element $K\in \cT_{h}$
\begin{equation}
\label{eqn:pi}
	-\Delta (\Pi w) =0, \mbox{ in } K,\qquad \Pi w = w, \mbox{ on } \partial K,
\end{equation}
for each $w\in V_{h}^{c}$; namely, $\Pi w$ is the harmonic extension of $w|_{\partial K} \in \B_{1}(\partial K).$ We remark that when $k=1$, it is not necessary to solve \eqref{eqn:pi} in order to construct the operator $\Pi.$ In fact, in this case the matrix representation of $\Pi$ is just the identity matrix due to the corresponding degrees of freedom for $V_{h}$ and $V_{h}^{c}$. This intergrid transfer operator defined in \eqref{eqn:pi} satisfies the following properties. 
\begin{lemma}
\label{lem:pi}
For any $w\in V_{h}^{c}$, the operator $\Pi : V_{h}^{c} \to V_{h}$ defined by \eqref{eqn:pi} satisfies the following properties: 
\begin{align}
		\|\Pi w\|_{A} &\lesssim \|w\|_{A_{c}} \label{eqn:pistab}\\
		\|(I - \Pi) w\|_{0,\kappa} &\lesssim h\|w\|_{A_{c}} \label{eqn:piapprox}. 
\end{align}
\end{lemma}
\begin{proof}
	By the stability assumption \eqref{eqn:bstab} we only need to show that 
	\[
		\|\nabla \Pi w\|_{0,K} \lesssim \|\nabla w\|_{0,K}, \quad \forall K\in \cT_{h}.
	\]
	By the definition of $\Pi w$ and trace inequality, we have 
	\begin{align*}
		\|\nabla \Pi w\|_{0,K} \lesssim \|w\|_{1/2, \partial K} \lesssim \|\nabla w\|_{0,K}.
	\end{align*}
	This implies \eqref{eqn:pistab}. 
	
	To prove \eqref{eqn:piapprox}, we can view $w\in V_{h}^{c}$ as piecewise linear interpolation of $\Pi w$ on each element $K\in \cT_{h}$. Therefore, by the stability of $\Pi$ \eqref{eqn:pistab} we have 
\[
	\|\Pi w - w\|_{0,K} \lesssim h \|\nabla \Pi w\|_{0,K} \lesssim h \|\nabla w\|_{0,K}.
\]
The estimate \eqref{eqn:piapprox} then follows.
\end{proof}
Inequality \eqref{eqn:pistab} is exactly the condition ({\bf C1}) in Theorem~\ref{thm:aux0}. It remains to verify the stable decomposition condition ({\bf C2}). This is achieved by the following lemma.
\begin{lemma}
\label{lm:stabdecomp}
There exists a linear operator $P: V_{h} \to V_{h}^{c}$ such that 
\begin{align}
	\|P v\|_{A_{c}} &\lesssim \|v\|_{A} \label{eqn:stab}\\
	\|v - \Pi P v\|_{0,\kappa} &\lesssim h \|v\|_{A} \label{eqn:approx}
\end{align}
\end{lemma}
\begin{proof}
	Let $P = I_{K}: V_{h}^{K} \to V_{h}^{c}(K)$ be the local interpolation operator on the subdomain $K\in \cT_{h}$ such that $I_{K}(v) (\bs x_{i}) = v(\bs x_{i})$ for each $v\in V_{h}^{K}$ and each vertex $\bs x_{i}\in K$. Note this local interpolation operator is not well-defined for general $H^{1}(K)$, but it is well-defined on the VEM space $V_{h}$. Then it is obvious that  
	\begin{equation}
	\label{eqn:Papprox}
		\|(I-P) v\|_{0,K} \lesssim h \|\nabla v\|_{0,K},
	\end{equation} and hence $P$ satisfies the stability \eqref{eqn:stab}.
	
	To prove \eqref{eqn:approx}, by triangle inequality
	\begin{align*}
		\|v - \Pi P v\|_{0,K} &\le \|v - Pv\|_{0,K} + \|(I-\Pi) Pv\|_{0,K}\\
		&\lesssim h \|\nabla v\|_{0,K} + h\| \nabla (Pv)\|_{0,K},
	\end{align*}
	where in the second inequality we used the approximation \eqref{eqn:Papprox} of $P$ for the first term, and we used \eqref{eqn:piapprox} for the second term.  The estimate \eqref{eqn:approx} then follows by \eqref{eqn:stab}. This completes the proof.
\end{proof}
Lemma~\ref{lm:stabdecomp} implies that for any $v\in V_{h}$, we can decompose it as $v = v_{0} + \Pi w,$ where $w = Pv \in V_{h}^{c}$ and $v_{0} = v - \Pi w \in V_{h}$. With the smoothing property \eqref{eqn:smoother2}, we get 
\[	
	s(v_{0}, v_{0}) + a(w, w) \lesssim h^{-2} \|v_{0}\|_{0,\kappa}^{2} + \|v\|_{A}^{2} \lesssim \|v\|_{A}^{2}.
\]
This verifies the condition ({\bf C2}). In summary, by the auxiliary space preconditioner framework Theorem~\ref{thm:aux0}, we have the following main theorem. 
\begin{theorem}
\label{thm:aux}
	The auxiliary space preconditioner (both additive and multiplicative) defined in \eqref{eqn:Badd}~\eqref{eqn:Bmul} satisfies: 
	$$
		\cK(BA) \le C,
	$$
	where the constant $C$ is independent of the mesh size $h$ and $\kappa$. 
\end{theorem}

Note that in general, the intergrid operator defined by \eqref{eqn:pi} is not surjective. In fact, there is no surjective intergrid operator from $V_{h}^{c}$ to $V_{h}$ when $k\ge 2$. Hence, when $k\ge 2$ it is necessary to use the auxiliary space preconditioner.  However, when $k=1$ we have the following corollary for the fictitious space preconditioner $B_{{\rm fict}}$. 
\begin{corollary}
\label{cor:fict}
	When $k=1$, namely the linear VEM, the fictitious space preconditioner $B_{{\rm fict}} = \Pi A_{c}^{-1} \Pi^{t}$ satisfies: 
	\[
		\cK(B_{{\rm fict}}A) \le C,
	\]
	where the constant $C$ is independent of the mesh size and $\kappa$. 
\end{corollary}
\begin{proof}
	When $k=1$, the operator $\Pi$ defined in \eqref{eqn:pi} is surjective and satisfies \eqref{eqn:pistab} and \eqref{eqn:stab}. Therefore, the conclusion follows immediately by the fictitious space lemma (cf. \cite[Theorem 2.2]{Hiptmair.R;Xu.J2007}). 
\end{proof}

\section{Numerical Experiments}
\label{sec:num}
In this section, we present several numerical experiments in both 2D and 3D to verify the result in Theorem~\ref{thm:aux} on the performance of the proposed preconditioners. In all these tests, we use 2-sweeps symmetric Gauss-Seidel smoother. The stopping criteria is $\|r_{k}\| / \|r_{0}\| <10^{-12}$ for the PCG algorithm, where $r_{k}= f-Au_{k}$ is the residual. For the coarse solver, we use the AMG algorithm implemented in $i$FEM~\cite{Chen.L2008}. 

\subsection{2D Examples}
In the first example, we consider the model problem \eqref{eqn:model} in the unit square $\Omega = [0,1]^{2}$ with constant coefficient $\kappa =1$. Figure~\ref{fig:poly2d} is an example of the polytopal mesh of the unit square domain (with 100 elements) generated using \mcode{PolyMesher} \cite{Talischi.C;Paulino.G;Pereira.A;Menezes.I2012}, and Figure~\ref{fig:tri2d} is the corresponding Delaunay triangular mesh. The VEM discretization is defined on the polytopal mesh (cf. Figure~\ref{fig:poly2d}), while the auxiliary space using the standard conforming $\P_{1}$ finite element discretization is defined on the corresponding triangular mesh (cf. Figure~\ref{fig:tri2d}). 
\begin{figure}[htbp]
\centering
	\parbox{0.45\textwidth}{
       \includegraphics[width=0.4\textwidth]{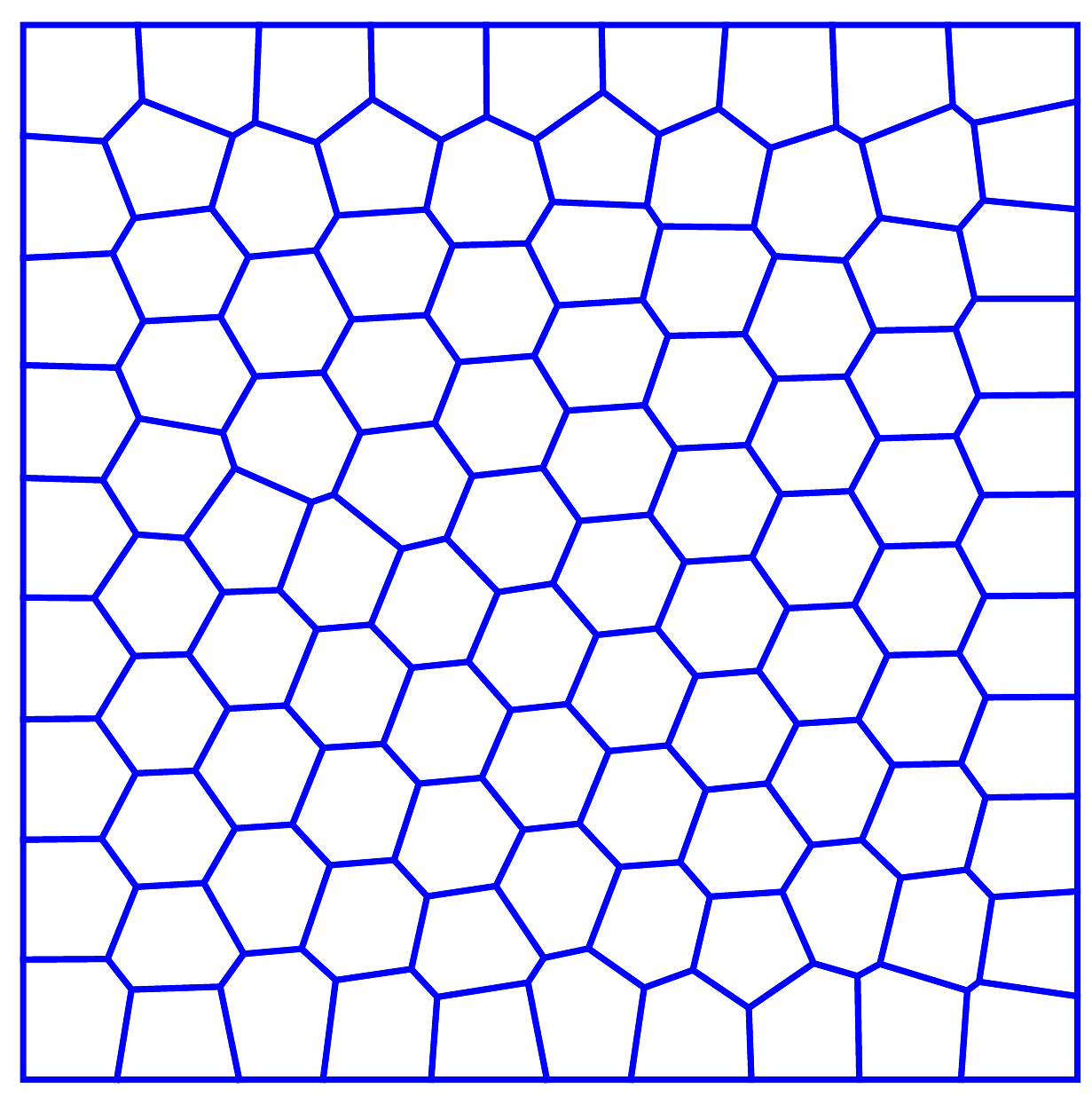}
       \caption{Polygonal Mesh $\cT_{h}$ of the Unit Square Domain (100 Elements)}
       \label{fig:poly2d}}
       \quad
       \begin{minipage}{0.45\textwidth}
       \includegraphics[width=0.89\textwidth]{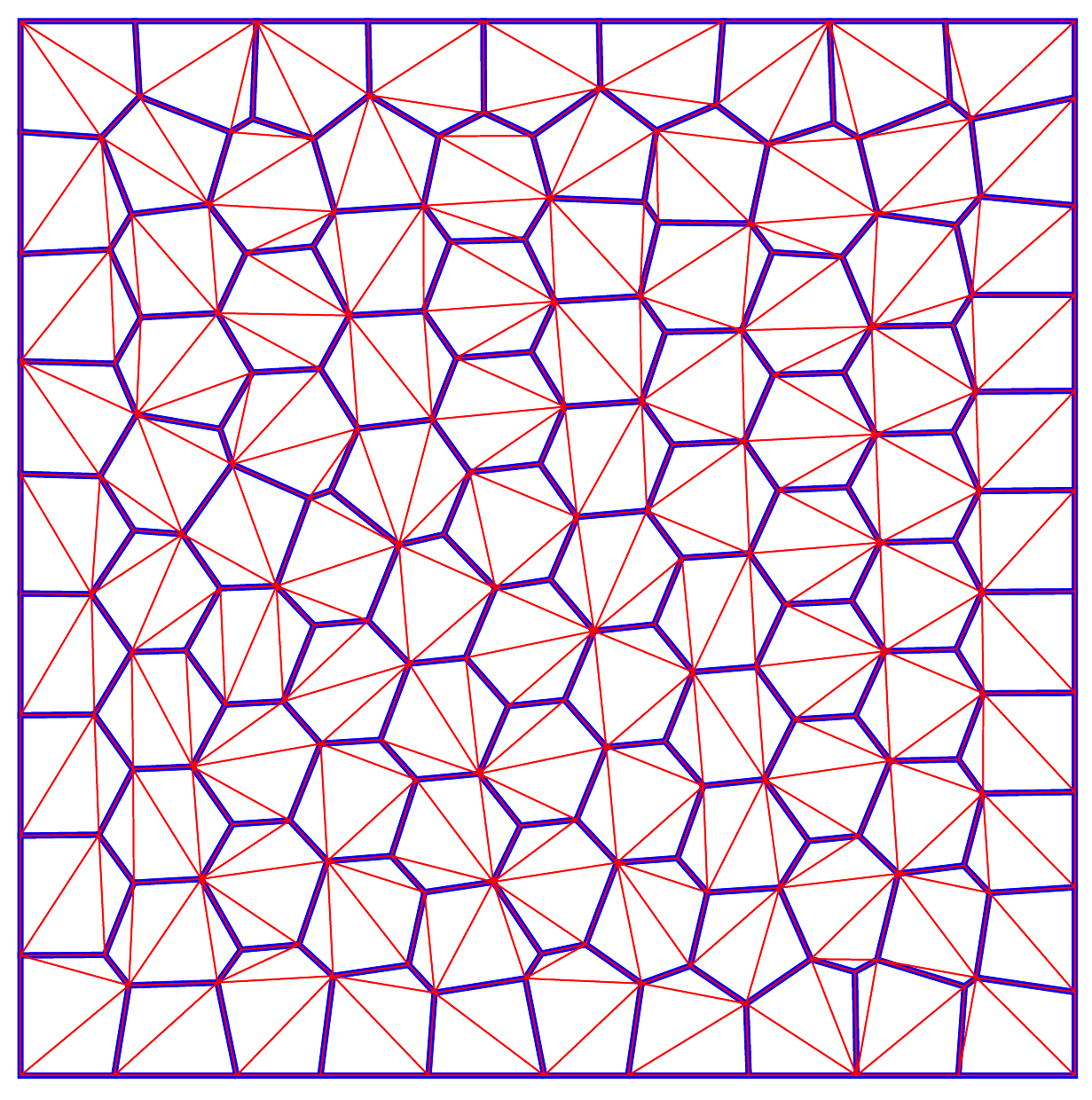}
       \caption{The Corresponding Delaunay Triangle Mesh $\cT_{h}^{c}$}
       \label{fig:tri2d}
       \end{minipage}
\end{figure}

Tables~\ref{tab:2d} shows the estimated condition numbers  (the number of PCG iterations) for the additive and multiplicative preconditioned systems. 
{\footnotesize
\begin{table}
\caption{Estimated condition numbers (number of PCG iterations) in 2D with constant coefficients.}
\begin{tabular}{c||c|c|c|c|c}
\hline
  \# Polytopal Elements & 10 & $10^{2}$   & $10^{3}$        & $10^{4}$        & $10^{5}$  
\\
\hline\hline
 $\mathcal{K}(A)$ & 3.45 (9)&  3.86e1 (41)&  3.80e2 (117)&  3.88e3 (351)& 4.07e4 (1100) \\\hline
 $\mathcal{K}(B_{\rm sgs}A)$ & 1.07(6) & 3.78 (15) &3.20e1 (37) & 3.17e2 (104) & 3.17e3 (318)\\\hline
 $\mathcal{K}(B_{\rm fict}A)$& 2.92 (8) & 5.75 (26) & 7.53 (29) & 8.73 (32) & 9.67(36) \\\hline
 $\mathcal{K}(B_{\rm add}A)$ &  1.53 (9)&  1.71 (14)&  1.94 (14)&  1.99 (14)& 2.00 (13) \\\hline
 $\mathcal{K}(B_{\rm mul}A)$ &  1.06 (8)&  1.21 (10)&  1.04 (7)&  1.02 (6)& 1.02 (6) \\\hline
\end{tabular}
\label{tab:2d}
\end{table}
}
For comparison, we also include the estimated condition numbers $\cK(A)$, $\cK(B_{{\rm sgs}}A)$ and $\cK(B_{{\rm fict}}A)$, where $B_{{\rm sgs}}$ is the (2-sweep) symmetric Gauss-Seidel preconditioner (same below) and $B_{{\rm fict}}$ is the fictitious space preconditioner using the conforming FEM.  As we can observe from this table, while the condition numbers $\cK(A)$, $\cK(B_{{\rm sgs}}A)$ increase as the mesh is refined. The condition number $\cK(B_{{\rm fict}}A)$ increase slightly as the mesh is refined. 
On the other hand, the condition numbers of $\mathcal{K}(B_{{\rm add}}A)$ and $\mathcal{K}(B_{{\rm mul}}A)$ are uniformly bounded. 

In the second test, we consider the problem with jump coefficients. The coefficients $\kappa$ are generated randomly on each polygon element (see Figure~\ref{fig:jump2d} for an example of the coefficient distribution with 100 elements, the integer in each polygonal element is the magnitude of the coefficient.). 

\begin{figure}[h]
\begin{center}
       \includegraphics[width=0.45\textwidth]{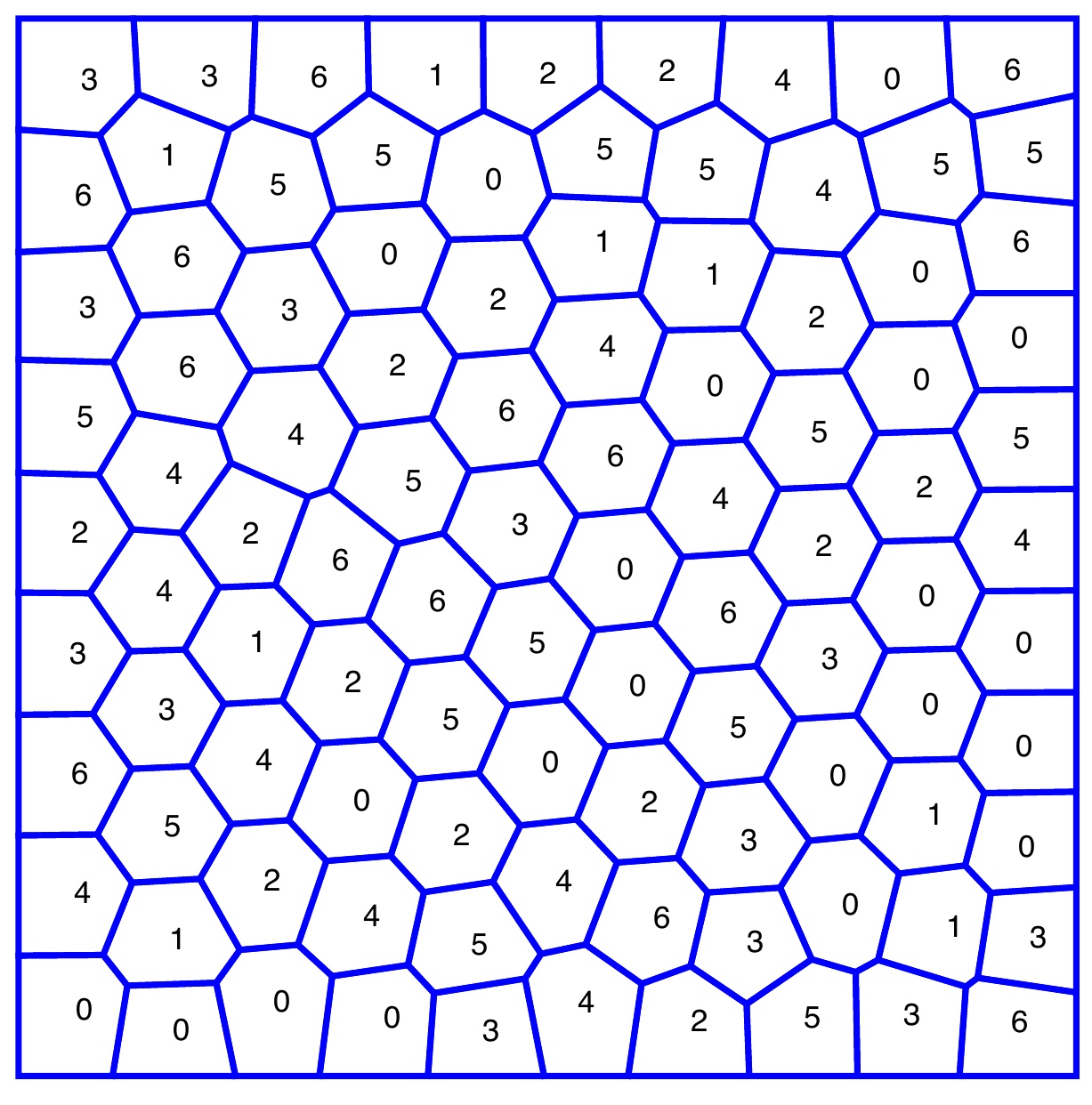}
       \caption{Random Jump Coefficients $10^{k}$ (100 Elements)}
       \label{fig:jump2d}
\end{center}
 \end{figure}
Note that the coefficient settings are different in different polytopal meshes. Tables~\ref{tab:2djump} shows the estimated condition numbers (the number of PCG iterations). Here, ``-'' means the PCG algorithm failed to converge after 1200 iterations.  As we can see from this table, while $\cK(A)$ and $\cK(B_{{\rm sgs}}A)$ increase dramatically, the condition number $\mathcal{K}(B_{{\rm add}}A)$ and $\mathcal{K}(B_{{\rm mul}}A)$ are nearly uniformly bounded. These observations verify the estimate given in Theorem~\ref{thm:aux}. 
{\footnotesize
\begin{table}
\caption{Estimated condition numbers (number of PCG iterations) in 2D with jump coefficients.}
\begin{tabular}{c||c|c|c|c|c}
\hline
  \# Polytopal Elements & 10 & $10^{2}$   & $10^{3}$        & $10^{4}$        & $10^{5}$  
\\
\hline\hline
 $\mathcal{K}(A)$                   & 2.44 (11)&  2.73e6 (578)&  - & - & - \\\hline
 $\mathcal{K}(B_{\rm sgs}A)$ & 1.18(5) & 3.90e2 (26) &3.93e3 (409) & - & -\\\hline
 $\mathcal{K}(B_{\rm fict}A)$& 3.27 (8) & 6.94 (33) & 6.42 (36) & 11.6 (44) & 13.6 (53) \\\hline
 $\mathcal{K}(B_{\rm add}A)$ &  1.54 (9)&  3.51 (20)&  3.60 (25)&  3.67 (25)& 3.80 (26) \\\hline
 $\mathcal{K}(B_{\rm mul}A)$ &  1.06 (6)&  1.74 (15)&  1.82 (16)&  1.84 (16)& 1.88 (17) \\\hline
\end{tabular}
\label{tab:2djump}
\end{table}
}

In the third test, we consider the performance of the preconditioners for voronoi meshes which violate the assumption ({\bf A}) (see Figure~\ref{fig:2dvoronoi} for an example of 100 polygons). As we can observe from this figure, the aspect ratio for some polygons are quite high -- thus the partition is no longer quasi-uniform. Similar to before, we use Delaunay triangulation of this mesh to construct the auxiliary space.
\begin{figure}[h]
\begin{center}
       \includegraphics[width=0.45\textwidth]{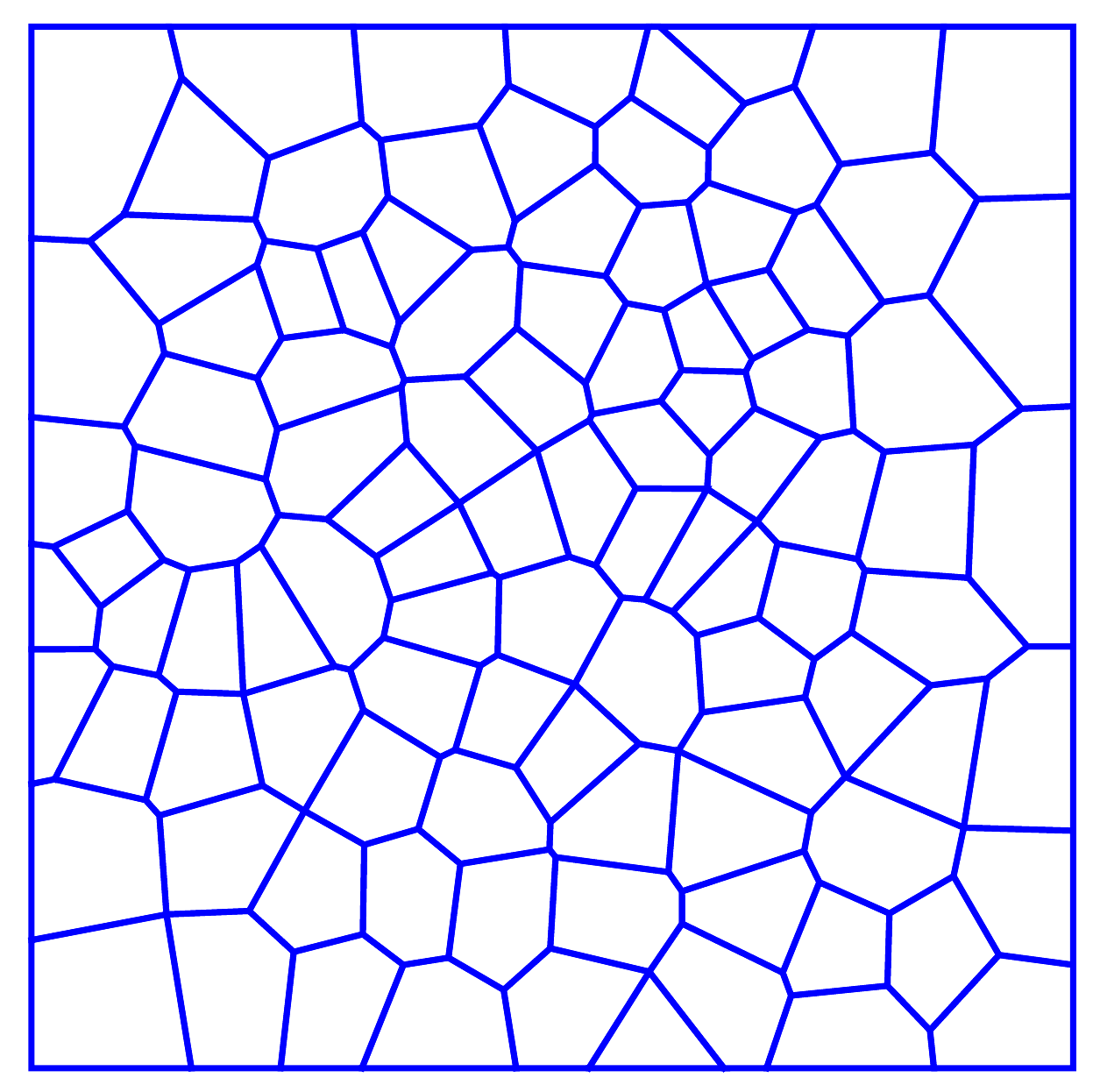}
       \caption{Voronoi mesh (100 Elements)}
       \label{fig:2dvoronoi}
\end{center}
 \end{figure}
 
 Table~\ref{tab:2dvoronoi} shows the estimated condition numbers with the number of PCG iterations for different preconditioners. As we can see from this table, both additive and multiplicative auxiliary space preconditioners are still robust with respect to the problem size. 
{\footnotesize
\begin{table}
\caption{Estimated condition numbers (number of PCG iterations) in 2D voronoi polygonal mesh.}
\begin{tabular}{c||c|c|c|c|c}
\hline
  \# Polytopal Elements & 10 & $10^{2}$   & $10^{3}$        & $10^{4}$        & $10^{5}$  
\\
\hline\hline
$\mathcal{K}(A)$& 4.76 (9) & 6.89e1 (52) & 6.59e2 (171) & 6.49e3 (537) & - \\\hline
$\mathcal{K}(B_{\rm sgs}A)$& 1.13 (6) & 4.93 (17) & 3.81e1 (45) & 3.57e2 (134) & 3.40e3 (400) \\\hline
$\mathcal{K}(B_{\rm fict}A)$& 4.66 (9) & 7.92 (34) & 2.04e1 (43) & 2.32e1 (46) & 1.62e1 (52) \\\hline
$\mathcal{K}(B_{\rm add}A)$& 1.58 (9) & 1.72 (16) & 3.09 (18) & 3.16 (19) & 1.91 (17) \\\hline
$\mathcal{K}(B_{\rm mul}A)$& 1.32 (11) & 2.25 (16) & 1.48 (13) & 1.29 (12) & 1.14 (10) \\\hline
\end{tabular}
\label{tab:2dvoronoi}
\end{table}
}

\subsection{3D Example}
Now we consider the model problem \eqref{eqn:model} in a 3D cubic domain $\Omega =[0,1]^{3}$. We subdivide the domain into hexagonal elements (cubes) with mesh size $h$ at each level. The VEM discretization is defined on the hexagon mesh. For the auxiliary space, we further divide each hexagon into six tetrahedrons to construct the auxiliary mesh and to define the $\P_{1}$ conforming finite element discretization on this auxiliary mesh (see for example, Figure~\ref{fig:jump3d}).


In this example, we test various discontinuous coefficient settings. Let $\Omega_{1} =[0.25,0.5]^{3}$ and $\Omega_{2} = [0.5,0.75]^{3}$ (see Figure~\ref{fig:jump3d}). We set the coefficient $\kappa|_{\Omega_{1}\cup\Omega_{2}}= \kappa_{1}= 10^{k}$ (with $k=-6, -4, -2, 0, 2, 4, 6$) and $\kappa|_{\Omega\setminus (\Omega_{1}\cup\Omega_{2})} = 1$.
\begin{figure}[h]
\begin{center}
       \includegraphics[width=0.45\textwidth]{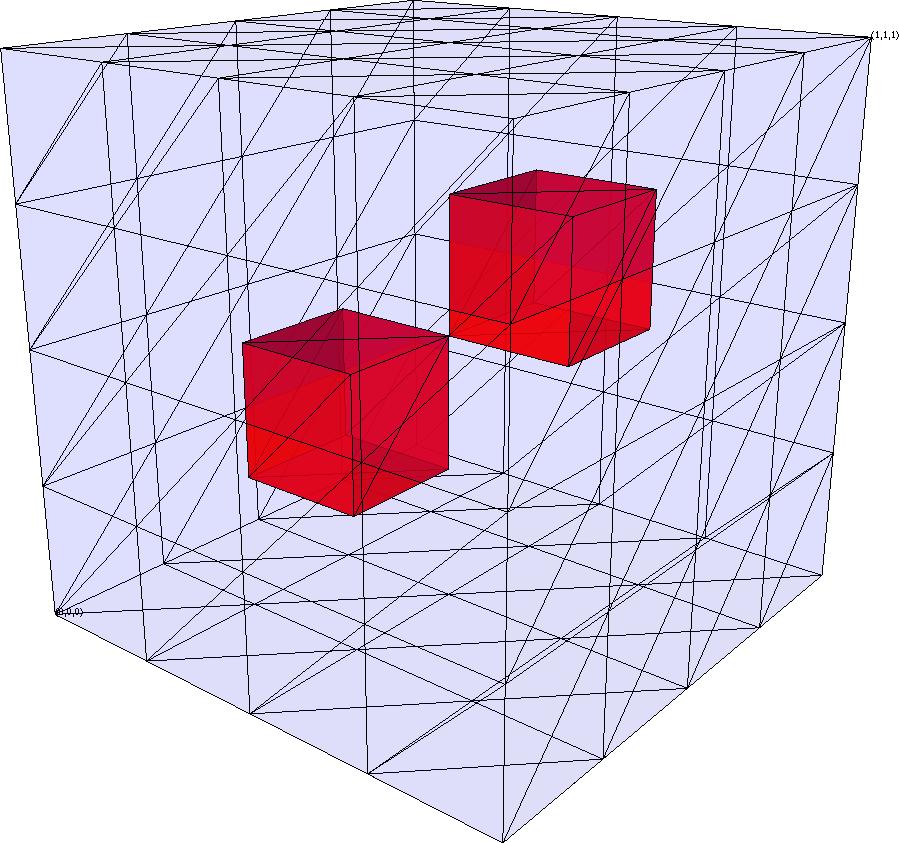}
       \caption{{\footnotesize 3D uniform mesh with jump coefficients}}
       \label{fig:jump3d}
\end{center}
 \end{figure}
Table~\ref{tab:3d} presents the estimated condition number of the preconditioned systems with respect to different choice of $\kappa_{1}$ and mesh size.
{\scriptsize
\begin{table}
\caption{Estimated condition numbers (number of PCG iterations) in 3D. The coefficient $\kappa|_{\Omega_{1}\cup\Omega_{2}} =\kappa_{1}= 10^{k}$ for various choices of $k$, and $\kappa|_{\Omega\setminus(\Omega_{1}\cup \Omega_{2})} =1.$ }
\begin{center}\begin{tabular}{c|c||c|c|c|c|c|c}
\hline
 $\kappa_{1}$ & $ h$ & $2^{-2}$ &  $2^{-3}$ &  $2^{-4}$ &  $2^{-5}$ &  $2^{-6}$ &  $2^{-7}$ \\
\hline\hline
\multirow{5}{*}{$10^{-6}$} & $\mathcal{K}(A)$& 1.15e6 (8) & 8.76e6 (28) & 6.94e7 (56) & 5.54e8 (110) & 4.43e9 (215) & 3.54e10 (420) \\
	&$\mathcal{K}(B_{\rm sgs}A)$& 1.10 (7) & 2.11 (12) & 6.54 (23) & 2.44e1 (43) & 9.57e1 (71) & 3.81e2 (118) \\
	&$\mathcal{K}(B_{\rm fict}A)$& 1.32 (6) & 1.44 (11) & 1.41 (10) & 1.39 (9) & 1.37 (8) & 1.33 (7) \\
	&$\mathcal{K}(B_{\rm add}A)$& 1.14 (10) & 1.35 (11) & 1.73 (15) & 1.92 (17) & 1.98 (17) & 1.99 (16) \\
	&$\mathcal{K}(B_{\rm mul}A)$& 1.00 (5) & 1.00 (5) & 1.00 (5) & 1.00 (5) & 1.00 (4) & 1.00 (4) \\\hline\hline
\multirow{5}{*}{$10^{-4}$} & $\mathcal{K}(A)$& 1.15e4 (7) & 8.76e4 (26) & 6.94e5 (51) & 5.54e6 (99) & 4.43e7 (194) & 3.54e8 (379) \\
	&$\mathcal{K}(B_{\rm sgs}A)$& 1.10 (7) & 2.11 (12) & 6.54 (23) & 2.44e1 (38) & 9.57e1 (63) & 3.81e2 (115) \\
	&$\mathcal{K}(B_{\rm fict}A)$& 1.32 (6) & 1.44 (11) & 1.41 (10) & 1.39 (9) & 1.37 (8) & 1.33 (7) \\
	&$\mathcal{K}(B_{\rm add}A)$& 1.14 (9) & 1.35 (11) & 1.73 (14) & 1.92 (15) & 1.98 (15) & 1.99 (13) \\
	&$\mathcal{K}(B_{\rm mul}A)$& 1.00 (5) & 1.00 (5) & 1.00 (5) & 1.00 (5) & 1.00 (4) & 1.00 (4) \\\hline\hline
\multirow{5}{*}{$10^{-2}$} & $\mathcal{K}(A)$& 1.15e2 (7) & 8.76e2 (24) & 6.94e3 (46) & 5.54e4 (90) & 4.43e5 (175) & 3.54e6 (346) \\
	&$\mathcal{K}(B_{\rm sgs}A)$& 1.10 (7) & 2.11 (12) & 6.54 (21) & 2.44e1 (35) & 9.57e1 (62) & 3.81e2 (115) \\
	&$\mathcal{K}(B_{\rm fict}A)$& 1.32 (6) & 1.45 (11) & 1.41 (10) & 1.39 (9) & 1.37 (8) & 1.28 (6) \\
	&$\mathcal{K}(B_{\rm add}A)$& 1.14 (8) & 1.36 (10) & 1.73 (13) & 1.92 (14) & 1.98 (14) & 1.99 (13) \\
	&$\mathcal{K}(B_{\rm mul}A)$& 1.00 (5) & 1.00 (5) & 1.00 (5) & 1.00 (4) & 1.00 (4) & 1.00 (4) \\\hline\hline
\multirow{5}{*}{1}  & $\mathcal{K}(A)$& 4.44 (6) & 1.74e1 (21) & 6.94e1 (40) & 5.54e2 (78) & 4.43e3 (153) & 3.54e4 (302)\\
	&$\mathcal{K}(B_{\rm sgs}A)$& 1.10 (7) & 2.11 (12) & 6.54 (21) & 2.44e1 (35) & 9.57e1 (62) & 3.81e2 (115) \\
	&$\mathcal{K}(B_{\rm fict}A)$& 1.32 (6) & 1.45 (11) & 1.41 (10) & 1.39 (9) & 1.33 (7) & 1.28 (6) \\
	&$\mathcal{K}(B_{\rm add}A)$& 1.14 (8) & 1.36 (10) & 1.73 (13) & 1.92 (14) & 1.98 (14) & 1.99 (13) \\
	&$\mathcal{K}(B_{\rm mul}A)$& 1.00 (5) & 1.00 (5) & 1.00 (5) & 1.00 (4) & 1.00 (4) & 1.00 (3) \\\hline\hline
\multirow{5}{*}{$10^{2}$} & $\mathcal{K}(A)$& 3.88e2 (6) & 2.00e2 (22) & 9.98e1 (44) & 2.77e2 (80) & 1.11e3 (143) & 4.43e3 (273) \\
	&$\mathcal{K}(B_{\rm sgs}A)$& 1.10 (7) & 2.11 (12) & 6.54 (21) & 2.44e1 (35) & 9.57e1 (62) & 3.81e2 (115) \\
	&$\mathcal{K}(B_{\rm fict}A)$& 1.32 (6) & 1.45 (11) & 1.41 (10) & 1.39 (9) & 1.33 (7) & 1.28 (6) \\
	&$\mathcal{K}(B_{\rm add}A)$& 1.14 (8) & 1.36 (10) & 1.73 (13) & 1.92 (14) & 1.98 (14) & 1.99 (13) \\
	&$\mathcal{K}(B_{\rm mul}A)$& 1.00 (5) & 1.00 (5) & 1.00 (5) & 1.00 (4) & 1.00 (4) & 1.00 (3) \\\hline\hline
\multirow{5}{*}{$10^{4}$} & $\mathcal{K}(A)$& 3.88e4 (6) & 2.00e4 (22) & 9.98e3 (47) & 5.00e3 (89) & 2.50e3 (163) & 4.43e3 (295) \\
	&$\mathcal{K}(B_{\rm sgs}A)$& 1.10 (7) & 2.11 (12) & 6.54 (21) & 2.44e1 (35) & 9.57e1 (62) & 3.81e2 (115) \\
	&$\mathcal{K}(B_{\rm fict}A)$& 1.32 (6) & 1.45 (11) & 1.41 (10) & 1.39 (9) & 1.33 (7) & 1.28 (6) \\
	&$\mathcal{K}(B_{\rm add}A)$& 1.14 (8) & 1.36 (10) & 1.73 (13) & 1.92 (14) & 1.98 (14) & 1.99 (13) \\
	&$\mathcal{K}(B_{\rm mul}A)$& 1.00 (5) & 1.00 (5) & 1.00 (5) & 1.00 (4) & 1.00 (4) & 1.00 (3) \\\hline\hline
\multirow{5}{*}{$10^{6}$} & $\mathcal{K}(A)$ & 3.88e6 (9) & 2.00e6 (22) & 9.99e5 (51) & 5.00e5 (96) & 2.50e5 (180) & 1.25e5 (331) \\
	&$\mathcal{K}(B_{\rm sgs}A)$& 1.10 (7) & 2.11 (12) & 6.54 (21) & 2.44e1 (35) & 9.57e1 (62) & 3.81e2 (115) \\
	&$\mathcal{K}(B_{\rm fict}A)$& 1.32 (6) & 1.45 (11) & 1.41 (10) & 1.39 (9) & 1.33 (7) & 1.28 (6) \\
	&$\mathcal{K}(B_{\rm add}A)$ & 1.14 (8) & 1.36 (10) & 1.73 (13) & 1.92 (14) & 1.98 (14) & 1.99 (13) \\
	&$\mathcal{K}(B_{\rm mul}A)$ & 1.00 (5) & 1.00 (5) & 1.00 (5) & 1.00 (4) & 1.00 (4) & 1.00 (3) \\\hline
\hline
\end{tabular}
\label{tab:3d}
\end{center}
\end{table}
}
As we can see from Table~\ref{tab:3d}, the condition number of $A$ depends on both the coefficient $\kappa$ and the mesh size. On the other hand, both the fictitious space preconditioner and the auxiliary space preconditioners (additive or multiplicative) are efficient and robust with respect to jumps in the coefficient $\kappa$ and the mesh size. These results justify Theorem~\ref{thm:aux} and Corollary~\ref{cor:fict}.

\section{Conclusion}
\label{sec:con}
In this paper, we developed the additive and multiplicative auxiliary space preconditioners for solving the linear system arising from the virtual element methods discretization on polytopal meshes for the second order elliptic equation with jump coefficients. We used an auxiliary simplicial triangulation to construct the coarse space. The auxiliary space preconditioners consist of a smoother and a coarse space correction. We showed that the condition numbers of the preconditioned systems are uniformly bounded, independent of the problem size and the jump in coefficients. For the linear VEM discretization, we also showed that the fictitious space preconditioner is also a robust and efficient preconditioner. Numerical experiments were presented to demonstrate the performance of these preconditioners.

In the analysis, we assumed that the polytopal mesh should satisfy ({\bf A}), that is, there exist a quasi-uniform auxiliary triangulation. However, the numerical experiments demonstrated that these preconditioners also perform uniformly for voronoi meshes which violate the quasi-uniform assumption. This indicates that we might be able to relax the quasi-uniformity condition. This is a topic of ongoing research. 

\section*{Acknowledgement}
This work is supported  in part by  NSF DMS 1319110. 

\bibliographystyle{abbrv}
\bibliography{../zhuLibrary/books,../zhuLibrary/library}

%

\end{document}